\documentclass[12pt]{amsart}
\usepackage{times,fullpage}
\usepackage{latexsym,amscd,amssymb}
\newtheorem{thm}{Theorem}
\newtheorem{prop}{Proposition}
\newtheorem{lem}{Lemma}

\newtheorem{prob}{Problem}

\theoremstyle{remark}
\newtheorem{rem}{Remark}

\theoremstyle{definition}

\newcommand{\C}{\mathbb{ C}}

\newcommand{\OO}{\mathcal{ O}}

\newcommand{\PP}{\mathbb{ P}}

\newcommand{\Z}{\mathbb{ Z}}

\title[Chern numbers and diffeomorphism types]{Chern numbers and diffeomorphism types of projective varieties}
\author{D.~Kotschick}
\address{Mathematisches Institut, Ludwig-Maximilians-Universit\"at M\"unchen,
Theresienstr.~39, 80333 M\"unchen, Germany}
\email{dieter@member.ams.org}
\date{December 20, 2007; MSC 2000: primary 57R20, secondary 14J30, 14J35, 32Q55}

\begin{document}

\maketitle

\begin{center}
{\it Herrn Prof.~Dr.~F.~Hirzebruch zum 80.~Geburtstag gewidmet.}
\end{center}

\bigskip

\begin{abstract}
In 1954 Hirzebruch asked which linear combinations of Chern numbers are topological invariants of 
smooth complex projective varieties. We give a complete answer to this question in small dimensions,
and also prove partial results without restrictions on the dimension.
\end{abstract}

\section{Introduction}

More than fifty years ago, Hirzebruch raised the question to what extent the Chern and 
Hodge numbers of projective algebraic manifolds are topologically invariant, 
see Problem~31 in~\cite{Hir1}. He noted that Chern numbers of almost complex 
manifolds are not topologically invariant simply because there are too many 
almost complex structures, even on a fixed manifold. Then, in 1959, Borel and 
Hirzebruch gave an example of a $10$-dimensional closed oriented manifold with 
two projective algebraic structures for which $c_{1}^{5}$ are different, 
see~\cite{BH} Section~24.11. Until 1987, when the commentary in~\cite{Hir2} was 
written, nothing further was discovered concerning this question. In particular, 
Hirzebruch wrote then that he did not know whether $c_{1}^{3}$ and $c_{1}^{4}$ 
are topological invariants of complex three- and four-folds respectively.

In this paper we prove that in complex dimension $3$ the only linear combinations
of the Chern numbers $c_1^3$, $c_1c_2$ and $c_3$ that are invariant under 
orientation-preserving diffeomorphims of simply connected projective algebraic manifolds 
are the multiples of the Euler characteristic $c_3$. In dimension $4$ the only linear 
combinations of Chern numbers that are invariant 
are the linear combinations of the 
Euler characteristic $c_4$ and of the Pontryagin numbers $p_1^2$ and  $p_2$.
We also prove some partial results in arbitrary dimensions.

These results stem from the fact that in complex dimension $2$ the Chern number 
$c_1^2$ is not invariant under orientation-reversing homeomorphisms; cf.~\cite{MAorient}. 
By suitable stabilisation of the counterexamples, we find enough examples at least in dimensions
$3$ and $4$ to detect the independent variation of all Chern numbers which are not 
combinations of the Euler and  Pontryagin numbers. Our results suggest a weaker form of 
Hirzebruch's problem, asking whether the topology determines the Chern 
numbers up to finite ambiguity. We have no counterexample to an affirmative 
answer to this weaker question in the projective algebraic case, although it is known 
to be false for non-K\"ahler complex manifolds; cf.~\cite{LeB}.

\section{Preliminary results}

\subsection{Complex surfaces}

For complex surfaces there are two Chern numbers, $c_{2}$ and 
$c_{1}^{2}$, which turn out to be diffeomorphism invariants even without 
assuming that the diffeomorphism is orientation-preserving with respect to
the orientations given by the complex structures:
\begin{thm}\label{t:BLMS}
    If two compact complex surfaces are diffeomorphic, then their 
    Chern numbers coincide.
    \end{thm}
\begin{proof}
    In this case $c_{2}$ is the topological Euler 
    characteristic $e$. By Wu's formula we have
    \begin{equation}\label{eq:Wu}
	c_{1}^{2}(X) = 2e(X)+p_{1}(X) \ .
	\end{equation}
    The first Pontryagin number is $3$ times the signature, and so 
    the right-hand side of~\eqref{eq:Wu} is invariant not just under
    orientation-preserving diffeomorphisms, but even under 
    orientation-preserving homotopy equivalences\footnote{All this was 
    known in 1954, and Hirzebruch~\cite{Hir1} remarked that the 
    Chern numbers of an algebraic surface are topological invariants 
    (of the underlying oriented manifold).}.
    
    Now suppose that two compact complex surfaces  are 
    orientation-reversing diffeomorphic, with respect to the 
    orientations defined by their complex structures. Then, using 
    Seiberg--Witten theory, I proved in 1995, see Theorem~2 
    of~\cite{BLMSorient}, that the signatures of these surfaces vanish. 
    Thus, their Chern numbers agree by~\eqref{eq:Wu}.
    \end{proof}
The statement about orientation-reversing diffeomorphisms concerns projective 
algebraic surfaces only, because, by the classification of complex surfaces, 
a complex surface with positive signature is always projective.

As we saw in the proof, $c_{1}^{2}$ is invariant under 
orientation-preserving homotopy equivalences, and under 
orientation-reversing diffeomorphisms. But it is not invariant under 
orientation-reversing homeomorphisms. Already in 1991, I had proved 
the following:
\begin{thm}[\cite{MAorient}]\label{t:MA}
There are infinitely many pairs of simply connected projective algebraic surfaces 
$X_{i}$ and $Y_{i}$ of non-zero signature which are orientation-reversing 
homeomorphic.
\end{thm}
The proof is based on geography results for surfaces of general type due to 
Persson and Chen. The surfaces $X_i$ and $Y_i$ are projective algebraic 
because they are of general type. They can be chosen to contain embedded 
holomorphic spheres, in which case they can not be orientation-reversing 
diffeomorphic, although they are orientation-reversing homeomorphic.
This was the motivation for the results of~\cite{BLMSorient} quoted in the proof of Theorem~\ref{t:BLMS}.
    
    By Wu's formula~\eqref{eq:Wu} the homeomorphic surfaces $X_i$ and $Y_i$
    have different $c_{1}^{2}$. Indeed, the homeomorphism in question 
    preserves the Euler number and reverses the sign of the signature, 
    so that~\eqref{eq:Wu} gives:
    \begin{equation}\label{eq:c12}
	c_{1}^{2}(Y_{i})=4e(X_{i})-c_{1}^{2}(X_{i}) \ .
	\end{equation}
Wu's formula~\eqref{eq:Wu} shows in particular that the unoriented homeomorphism 
type almost determines the Chern numbers of a compact complex surface: 
there are only two possible values for $c_{1}^{2}$ (and only one for 
$c_{2}$, of course).

We shall use the examples from Theorem~\ref{t:MA} as building blocks for our 
high-dimensional examples.

\subsection{Inductive formulae for Chern classes}

We require the following easy calculation.
\begin{lem}\label{l:products}
    Let $A$ be a compact complex $n$-fold, and $B=A\times\C P^{1}$. 
    Then the Chern numbers of $B$ are 
	\begin{equation}\label{eq:Chern2}
	    c_{r_{1}}\ldots c_{r_{k}}(B) = 2\sum_{j=1}^{k}c_{r_{1}}\ldots 
	    c_{r_{j}-1}\ldots c_{r_{k}}(A) \ .
	    \end{equation}
    \end{lem}
\begin{proof}
    The Whitney sum formula $c(TB)=c(TA)c(T\C P^{1})$ for the total Chern classes implies that,
    with respect to the K\"unneth decomposition of the cohomology of $B$, the Chern classes of $B$ are 
	\begin{equation*}
	    \begin{split}
	    c_{1}(B) &= c_{1}(A)+c_{1}(\C P^{1})\\
	    c_{2}(B) &= c_{2}(A)+c_{1}(A)c_{1}(\C P^{1})\\
	    &\ldots\\
	    c_{n}(B) &= c_{n}(A)+c_{n-1}(A)c_{1}(\C P^{1})\\
	    c_{n+1}(B) &= c_{n}(A)c_{1}(\C P^{1}) \ .
	    \end{split}
	    \end{equation*}
	    The claim follows using that the first Chern number of $\C P^{1}$ equals $2$.
    \end{proof}
We also need the following generalization of Lemma~\ref{l:products} to non-trivial $\C P^1$-bundles:
\begin{lem}
Let $B$ be a compact complex surface and $E\longrightarrow B$ a holomorphic vector bundle
of rank two. Then the projectivisation $X=\PP(E)$ has 
\begin{equation}\label{eq:proj}
\begin{split}
c_3(X) &=2c_2(B) \ ,\\ 
c_1c_2(X) &=2(c_1^2(B)+c_2(B)) \ ,\\
c_1^3(X) &=6c_1^2(B)+2p_1(\PP(E)) \ .
\end{split}
\end{equation}
\end{lem}
Here $p_1(\PP(E))=c_1^2(E)-4c_2(E)$ is the first Pontryagin number for the group $SO(3)=PU(2)$, which 
is the structure group of the sphere bundle $X\longrightarrow B$. Notice that in the case that $p_1(\PP(E))=0$,
the formulae reduce to those obtained for the trivial bundle.
\begin{proof}
The formulae for $c_3$ and for $c_1c_2$ are immediate from the multiplicativity of the topological Euler characteristic and 
of the Todd genus, recalling that the Todd genera in dimension $2$ and $3$ are $\frac{1}{12}(c_1^2+c_2)$,
respectively $\frac{1}{24}c_1c_2$. To compute $c_1^3$ note that by the Leray-Hirsch theorem the cohomology ring of $X$ is 
generated as a $H^*(B)$-module by a class $y\in H^2(X)$ restricting as a generator to every fiber and satisfying the relation
$$
y^2+c_1(E)y+c_2(E) = 0 \ .
$$
Moreover, $c_1(X)=c_1(B)+c_1(E)+2y$ because the vertical tangent bundle has first Chern class $c_1(E)+2y$.
The third power is computed straightforwardly using the relation and the fact that $y$ evaluates to $1$ on the fiber.
\end{proof}

\section{Complex three-folds}

A variant of Hirzebruch's problem for three-folds was taken up by LeBrun 
in 1998, see~\cite{LeB}, who proved that there are closed $6$-manifolds 
which admit complex structures with different $c_{1}c_{2}$ and $c_{1}^{3}$. 
He even proved that a fixed manifold can have complex structures 
realising infinitely many different values for $c_{1}c_{2}$. However, 
for all the examples discussed in~\cite{LeB} only one of the complex 
structures is projective algebraic, or at least K\"ahler, and all the 
others are non-K\"ahler. Therefore, these examples say nothing about 
the topological invariance of Chern numbers for projective algebraic 
three-folds.

Nevertheless, both $c_{1}c_{2}$ and $c_{1}^{3}$ are not diffeomorphism invariants of projective 
three-folds:
\begin{prop}\label{t:3folds}
    There are infinitely many pairs of projective algebraic 
    three-folds $Z_{i}$ and $T_{i}$ with the following properties:
    \begin{enumerate}
	\item[(i)] Each $Z_{i}$ and $T_{i}$ admits an orientation-reversing 
	diffeomorphism.
	\item[(ii)] For each $i$ the manifolds underlying $Z_{i}$ and $T_{i}$ are 
	diffeomorphic.
	\item[(iii)] For each $i$ one has $c_{1}c_{2}(Z_{i})\neq c_{1}c_{2}(T_{i})$ 
	and $c_{1}^{3}(Z_{i})\neq c_{1}^{3}(T_{i})$.
	\end{enumerate}
    \end{prop}
    \begin{proof}
    Let $X_{i}$ and $Y_{i}$ be the algebraic surfaces from 
    Theorem~\ref{t:MA}, constructed in~\cite{MAorient}, and take 
    $Z_{i}=X_{i}\times\C P^{1}$ and $T_{i}=Y_{i}\times\C P^{1}$. Then 
    the identity on the first factor times complex conjugation on the 
    second factor gives an orientation-reversing selfdiffeomorphism of 
    $Z_{i}$ and of $T_{i}$.
    
    Denote by $\bar Y_{i}$ the smooth manifold underlying $Y_{i}$, 
    but endowed with the orientation opposite to the one induced by 
    the complex structure. Then $X_{i}$ and $\bar Y_{i}$ are 
    orientation-preserving homeomorphic simply connected smooth 
    four-manifolds, and are therefore h-cobordant. If $W$ is an 
    h-cobordism between them, then $W\times S^{2}$ is an h-cobordism 
    between $Z_{i}$ and $\bar T_{i} = \bar Y_{i} \times \C P^{1}$. By 
    Smale's h-cobordism theorem, $Z_{i}$ and $\bar T_{i}$ are 
    orientation-preserving diffeomorphic. As $Z_{i}$ and $T_{i}$ admit 
    orientation-reversing diffeomorphisms, we conclude that they are both 
    orientation-preserving and orientation-reversing diffeomorphic.
    
    For the Chern numbers~\eqref{eq:Chern2} gives 
    \begin{equation}\label{eq:Chern3}
	\begin{split}
    c_{1}c_{2}(Z_{i}) &=2(c_{1}^{2}+c_{2})(X_{i}) \ , \\
    c_{1}^{3}(Z_{i}) &=6c_{1}^{2}(X_{i}) \ ,
    \end{split}
    \end{equation}
    and similarly for $T_{i}$ and $Y_{i}$. As $X_{i}$ and $Y_{i}$ 
    have the same $c_{2}$ but different $c_{1}^{2}$, we conclude that 
    $Z_{i}$ and $T_{i}$ have different $c_{1}c_{2}$ and different 
    $c_{1}^{3}$.
    \end{proof}
Thus $c_{1}c_{2}$ and $c_{1}^{3}$ are not topological invariants 
of projective three-folds, but it is not yet clear that they vary 
independently. This is the content of the following:
\begin{thm}\label{tt:3folds}
The only linear combinations of the Chern numbers $c_1^3$, $c_1c_2$ and $c_3$ that 
are invariant under orientation-preserving diffeomorphisms of simply connected projective 
algebraic three-folds are the multiples of the Euler characteristic $c_3$. 
\end{thm}
\begin{proof}
First of all, let us dispose of the orientation question. If two complex three-folds are orientation-reversing
diffeomorphic with respect to the orientations given by their complex structures, then they become
orientation-preserving diffeomorphic after we replace one of the complex structures by its complex
conjugate. As the conjugate complex structure has the same Chern numbers as the original one, we
do not have to distinguish between orientation-preserving and orientation-reversing diffeomorphisms.

All the examples constructed in the proof of Proposition~\ref{t:3folds} have the property that 
$3c_{1}c_{2}-c_{1}^{3}$ agrees on $Z_{i}$ and $T_{i}$, as follows by combining~\eqref{eq:c12} 
with~\eqref{eq:Chern3} and the topological invariance of $c_{2}$ for surfaces. 
But, by Proposition~\ref{t:3folds}, linear combinations of $3c_{1}c_{2}-c_{1}^{3}$ 
and of $c_{3}$ are the only candidates left for combinations of Chern numbers 
that can be topological invariants of projective three-folds.
In order to show that $3c_{1}c_{2}-c_{1}^{3}$ is not an oriented diffeomorphism invariant
we shall use certain ruled manifolds which are non-trivial $\C P^1$-bundles, rather than
the products used above. 

Consider again a pair $X_i$ and $Y_i$ of simply connected algebraic surfaces as in Theorem~\ref{t:MA}.
For simplicity we just denote them by $X$ and $Y$, with orientations implicitly given by the complex 
structures. The oriented manifolds $X$ and $\bar Y$ are orientation-preserving h-cobordant. Let $M$ be the 
projectivisation of the holomorphic tangent bundle $TY$ of $Y$. 
Temporarily ignoring the complex structure of $M$, we think of it as a smooth oriented two-sphere bundle over $Y$,
or over $\bar Y$. If $W$ is any h-cobordism between $X$ and $\bar Y$, then the two-sphere bundle $M\longrightarrow\bar Y$
extends to a uniquely defined oriented two-sphere bundle $V\longrightarrow W$. Let $N$ be the restriction of this
bundle to $X\subset W$. If we give $N$ the orientation induced from that of $X$ and $M$ the orientation induced
from that of $\bar Y$, then, by construction, $V$ is an h-cobordism between $N$ and $M$. By Smale's h-cobordism
theorem, $M$ and $N$ are diffeomorphic.

Because the bundle $p\colon M\longrightarrow Y$ was defined as the projectivisation of the holomorphic tangent 
bundle of $Y$, its characteristic classes are $w_2(p)=w_2(Y)$ and $p_1(p)=c_1^2(Y)-4c_2(Y)$. Considered
as a bundle over $\bar Y$, $p$ has the same Stiefel--Whitney class, but the first Pontryagin number changes 
sign. It follows that $q\colon N\longrightarrow X$ has 
$$
p_1(q)=-c_1^2(Y)+4c_2(Y)=c_1^2(X) \ ,
$$
where the last equality is from~\eqref{eq:c12}. Moreover, $w_2(q)=w_2(X)$, although $X$ and $\bar Y$ are not
diffeomorphic. This follows for example from the cohomological characterisation of $w_2(X)$ as the unique 
element of $H^2(X;\Z_2)$ which for all $x$ satisfies
\[
w_2(X)\cdot x\equiv x^2 \pmod 2 \ .
\]

The bundle $q$ is determined by $w_2(q)=w_2(X)$ and $p_1(q)=c_1^2(X)$, and so we can think of it as the 
projectivisation of the holomorphic rank two bundle $\OO(K)\oplus\OO\longrightarrow X$. Therefore the total 
space $N$ inherits a complex-algebraic structure from that of $X$. Its Chern numbers are given by~\eqref{eq:proj}:
\begin{equation}
\begin{split}
c_3(N) &=2c_2(X) \ ,\\ 
c_1c_2(N) &=2(c_1^2(X)+c_2(X)) \ ,\\
c_1^3(N) &=8c_1^2(X) \ .
\end{split}
\end{equation}

This $N$ is diffeomorphic to $M$, which has a complex-algebraic structure as the projectivisation of the 
holomorphic tangent bundle of $Y$. (Recall from the beginning of the proof that we do not have to keep
track of the orientations induced by complex structures, because we can always replace a structure by
its complex conjugate.) The Chern numbers of $M$ are also given by~\eqref{eq:proj}:
\begin{equation}
\begin{split}
c_3(M) &=2c_2(Y) =2c_2(X) \ ,\\ 
c_1c_2(M) &=2(c_1^2(Y)+c_2(Y)) =2(-c_1^2(X)+5c_2(X)) \ ,\\
c_1^3(M) &=8c_1^2(Y)-8c_2(Y) =8(-c_1^2(X)+3c_2(X)) \ ,
\end{split}
\end{equation}
using~\eqref{eq:c12} to replace the Chern numbers of $Y$ by combinations of those of $X$.
Unlike for the examples in Proposition~\ref{t:3folds}, the combination $3c_1c_2-c_1^3$ is not the same 
for $M$ and $N$. This finally shows that $c_1c_2$ and $c_1^3$ vary independently (within a fixed
diffeomorphism type).
\end{proof}

Although the Chern numbers of a projective three-fold are not determined by the underlying differentiable manifold, 
this may still be the case up to finite ambiguity. By the Hirzebruch--Riemann--Roch theorem one has
\begin{equation}\label{eq:3RR}
\frac{1}{24}c_{1}c_{2} = 1 - h^{1,0} + h^{2,0} - h^{3,0} \ ,
\end{equation}
so that in the K\"ahler case $c_{1}c_{2}$ is bounded from above and 
from below by linear combinations of Betti numbers. In particular, 
for K\"ahler structures on a fixed $6$-manifold $c_{1}c_{2}$ can take 
at most finitely many values. We are left with the following:
\begin{prob}\label{prob2}
    Does $c_{1}^{3}$ take on only finitely many values on the projective 
    algebraic structures with the same underlying $6$-manifold?
    \end{prob}
The issue here is that there is no Riemann--Roch type formula 
expressing $c_{1}^{3}$ as a combination of Hodge numbers and the 
other Chern numbers.

For three-folds with ample canonical bundle one has $c_{1}^{3}<0$, and Yau's celebrated work~\cite{Y} 
gives $c_{1}^{3}\geq \frac{8}{3}c_{1}c_{2}$. As $c_{1}c_{2}$ is bounded below by a linear combination of 
Betti numbers, we have a positive answer to Problem~\ref{prob2} for this restricted class of projective three-folds. 
Even in the non-K\"ahler category, there are no examples where infinitely many values are known to arise
for $c_1^3$. 

\section{Higher dimensions}

It is now very easy to show that, except for the Euler number, no 
Chern number is diffeomorphism-invariant:
\begin{thm}\label{t:higher}
    For projective algebraic $n$-folds with $n\geq 3$ the only Chern 
    number $c_{I}$ which is diffeomorphism-invariant is the Euler number 
    $c_{n}$.
    \end{thm}
Note that by Theorem~\ref{t:BLMS} this is false for $n=2$, because in 
that case $c_{1}^{2}$ is also diffeomorphism-invariant. On the other 
hand, by Theorem~\ref{t:MA} it is not homeomorphism-invariant, so that 
Theorem~\ref{t:higher} is true for $n=2$ if we replace 
diffeomorphism-invariance by homeomorphism-invariance. As in the case 
of Proposition~\ref{t:3folds}, the examples we exhibit in the proof of 
Theorem~\ref{t:higher} admit orientation-reversing diffeomorphisms, 
so that one cannot restore diffeomorphism-invariance of $c_{I}\neq 
c_{n}$ by restricting to orientation-preserving diffeomorphisms only.
\begin{proof}
    For $n=3$ this was already proved in Proposition~\ref{t:3folds}. For 
    $n>3$ we take the examples $T_{i}$ and $Z_{i}$ from Proposition~\ref{t:3folds} 
    and multiply them by $n-3$ copies of $\C P^{1}$. Call these products 
    $T_{i}'$ and $Z_{i}'$. Using formula~\eqref{eq:Chern2} and induction, 
    we see that, on the one hand, $c_{n}$ is a universal multiple of the 
    $c_{2}$ of the surfaces we started with. On the other hand, 
    $c_{1}^{n}(T_{i}')$ and $c_{1}^{n}(Z_{i}')$ are universal multiples of 
    $c_{1}^{2}(X_{i})$ and of $c_{1}^{2}(Y_{i})$ respectively, and so 
    are different. All other Chern numbers $c_{I}$ are universal 
    linear combinations of $c_{2}(X_{i})$ and $c_{1}^{2}(X_{i})$, 
    respectively $c_{2}(Y_{i})$ and $c_{1}^{2}(Y_{i})$, with 
    the coefficients of both $c_{2}$ and $c_{1}^{2}$ strictly positive. 
    As $X_{i}$ and $Y_{i}$ have the same $c_{2}$ but different 
    $c_{1}^{2}$, the result follows.
    \end{proof}

Although the individual Chern numbers are not diffeomorphism-invariant, certain 
linear combinations are invariant once we restrict to orientation-preserving 
diffeomorphisms. Of course, as remarked by Hirzebruch~\cite{Hir2}, the Pontryagin 
numbers $p_{J}$ have this invariance property\footnote{Note that, unlike the 
Euler number, the Pontryagin numbers change sign under a change of orientation.}, 
but this only helps when the complex dimension is even.

\begin{prob}
Prove that, in arbitrary dimensions, the only combinations of Chern numbers
that are invariant under orientation-preserving diffeomorphisms of smooth complex projective 
varieties are linear combinations of Euler and Pontryagin numbers.
\end{prob}
For complex dimension $3$ this is Theorem~\ref{tt:3folds} above. 
Theorem~\ref{tt:4folds} below deals with the case of complex dimension $4$.

It would be interesting to know whether each of the Chern numbers 
$c_{I}\neq c_{n}$ takes on only finitely many values on a fixed smooth manifold. 
In the K\"ahler case this is known to be true for $c_{1}c_{n-1}$ by a result of 
Libgober and Wood, who showed that this Chern number is always a linear 
combination of Hodge numbers, see Theorem~3 in~\cite{LW}. In the non-K\"ahler 
case $c_{1}c_{n-1}$ can take on infinitely many values on a fixed manifold. 
This follows as in the proof of Theorem~\ref{t:higher} by taking products 
of LeBrun's examples~\cite{LeB} mentioned in the previous section with $\C P^{1}$, and 
using formula~\eqref{eq:Chern2}. Because $c_{1}c_{2}$ takes on infinitely many 
values on a fixed $6$-manifold, the same conclusion holds for $c_{1}c_{n-1}$ 
in real dimension $2n\geq 8$.

Returning to the K\"ahler case, the Riemann--Roch theorem expresses the 
$\chi_{p}$-genus\footnote{Our notation is consistent with~\cite{S}, 
changing the traditional superscript in $\chi^{p}$ from~\cite{TMAG} 
to a subscript.}
$$
\chi_{p} = \sum_{q=0}^{n} (-1)^{q}h^{p,q}
$$
as a linear combination of Chern numbers, and it follows that the 
combinations of Chern numbers which appear in this way can take on 
only finitely many values on a fixed manifold, as they are bounded 
above and below by linear combinations of Betti numbers.

\begin{rem}
    If the complex dimension $n$ is odd, then the Todd genus expressing 
    the Euler characteristic $\chi_{0}=(-1)^{n}\chi_{n}$ of the structure 
    sheaf as a combination of Chern numbers does not involve $c_{1}^{n}$. 
    This follows from the {\it Bemerkungen} in Section~1.7 of~\cite{TMAG}. 
    On the one hand, in any dimension, the coefficient of $c_{n}$ in the 
    Todd genus agrees with the coefficient of $c_{1}^{n}$. On the other 
    hand, the Todd genus is divisible by $c_{1}$ if $n$ is odd.
    \end{rem}
Generalising this Remark, and what we saw for $n=3$ in the previous 
section, we now prove:
\begin{prop}\label{t:S}
    If $M$ is K\"ahler of odd complex dimension $n>1$, then all 
    $\chi_{p}$ are linear combinations of Chern numbers which do not 
    involve $c_{1}^{n}$.
    \end{prop}
This shows that for odd $n$ there is no general way to extract the value 
of $c_{1}^{n}$ from the Hodge numbers. In particular, one can not 
obtain a finiteness result for $c_{1}^{n}$ in this way.
\begin{proof}
    The K\"ahler symmetries imply $\chi_{p}=(-1)^{n}\chi_{n-p}$, so 
    that it is enough to prove the claim for $p>\frac{n}{2}$. We shall 
    do this by descending induction starting at $p=n$.
    
    Salamon~\cite{S} proved that for $2\leq k\leq n$ the number
    \begin{equation}\label{eq:S}
    \sum_{p=k}^{n}(-1)^{p}{p\choose k}\chi_{p}
    \end{equation}
    is a linear combination of Chern numbers each of which involves a 
    $c_{i}$ with $i>n-2[\frac{k}{2}]$, see~\cite{S} Corollary~3.3.
    Using this for $n$ odd and $k=n$, we obtain once more the claim for 
    $\chi_{n}$ treated already in the Remark above.
    
    Suppose now that the claim has been proved for $\chi_{n}$, 
    $\chi_{n-1}$, \ldots, $\chi_{j}$ with $j>\frac{n}{2}+1$. Then we 
    consider~\eqref{eq:S} with $k=j-1$. (Note that this still satisfies 
    $k\geq 2$.) As $\chi_{p}$ with $p\geq j$ does not involve 
    $c_{1}^{n}$ by the induction hypothesis, Salamon's result implies 
    that $\chi_{j-1}$ does not involve $c_{1}^{n}$ either.
    \end{proof}
    
\section{Four-folds}    
    
In the case of four-folds, in addition to the Euler number $c_{4}$, the 
following are invariants of the underlying oriented smooth manifold:
\begin{equation}
    \begin{split}
	p_{1}^{2} &= (c_{1}^{2}-2c_{2})^{2} = 
	c_{1}^{4}-4c_{1}^{2}c_{2}+4c_{2}^{2}\\
	p_{2} &= c_{2}^{2}-2c_{1}c_{3}+2c_{4} \ .
	\end{split}
    \end{equation}
The vector space of Chern numbers of four-folds is $5$-dimensional, 
containing the $3$-dimensional subspace spanned by $c_{4}$, 
$p_{1}^{2}$ and $p_{2}$. It turns out that all combinations of Chern 
numbers that are invariant under orientation-preserving diffeomorphisms
are contained in this subspace:
\begin{thm}\label{tt:4folds}
The only linear combinations of Chern numbers that are invariant under 
orientation-preserving diffeomorphisms of simply connected projective 
algebraic four-folds are linear combinations of the Euler characteristic and
of the Pontryagin numbers. 
\end{thm}
\begin{proof}
This is a rather formal consequence of our results for complex three-folds.
Consider the vector space of Chern number triples $(c_3,c_2c_1,c_1^3)$.
Whenever we have a smooth six-manifold with two different complex structures,
the difference of the two Chern vectors must be in the kernel of any linear 
functional corresponding to a topologically invariant combination of Chern numbers.
In the proof of Theorem~\ref{tt:3folds} we produced two kinds of examples for
which these difference vectors were linearly independent. Therefore the space of 
topologically invariant combinations of Chern numbers is at most one-dimensional,
and as it contains $c_3$ it is precisely one-dimensional.

Consider now the four-folds obtained by multiplying the three-dimensional 
examples by $\C P^1$. If the difference of Chern vectors in a three-dimensional
example is $(0,a,b)$, then by~\eqref{eq:Chern2} the difference of Chern vectors
$(c_4,c_1c_3,c_2^2,c_1^2c_2,c_1^4)$ for the product with $\C P^1$ is 
$(0,2a,4a,4a+2b,8b)$. Two examples in dimension three with linearly
independent difference vectors lead to examples in dimension four which 
also have linearly independent difference vectors. Thus, in the five-dimensional 
space spanned by the Chern numbers of complex projective four-folds, the subspace invariant
under orientation-preserving diffeomorphisms has codimension at least two.
As it contains the linearly independent elements $c_4$, $p_1^2$ and $p_2$, it
is exactly three-dimensional.
\end{proof}

Concerning the weaker question which Chern numbers of projective or 
K\"ahler four-folds are determined by the topology up to finite ambiguity, 
this is so for $c_{1}c_{3}$ on general grounds, see the discussion above 
and~\cite{LW,S}. The formula for $p_{2}$ then shows that $c_{2}^{2}$ 
is also determined up to finite ambiguity. Using either the formula 
for $p_{1}^{2}$ or the Riemann--Roch formula for the structure sheaf, 
we conclude that $c_{1}^{4}-4c_{1}^{2}c_{2}$ is also determined up to 
finite ambiguity, but it is not clear whether this is true for 
$c_{1}^{4}$ and $c_{1}^{2}c_{2}$ individually. Note that a negative 
answer to Problem~\ref{prob2}, giving infinitely many values for 
$c_{1}^{3}$ on a fixed $6$-manifold, would show that $c_{1}^{4}$ also 
takes on infinitely many values on a fixed $8$-manifold by taking products 
with $\C P^{1}$.

For four-folds with ample canonical bundle one has $c_{1}^{4}>0$, and Yau's 
work~\cite{Y} gives $c_{1}^{4}\leq \frac{5}{2}c_{1}^{2}c_{2}$. Therefore
$$
0<c_{1}^{4}\leq\frac{5}{3}(4c_{1}^{2}c_{2}-c_{1}^{4}) \ .
$$
As the right-hand side takes on only finitely many values, the same is 
true for $c_{1}^{4}$, and then for $c_{1}^{2}c_{2}$ as well.

\begin{rem}
    Pasquotto~\cite{P} recently raised the question of 
    the topological invariance of Chern numbers of symplectic 
    manifolds, particularly in (real) dimensions $6$ and $8$. Our 
    results for K\"ahler manifolds of course show that Chern numbers 
    of symplectic manifolds are not topological invariants. In the 
    K\"ahler case we have used Hodge theory to argue that the 
    variation of Chern numbers is quite restricted, often to finitely 
    many possibilities. It would be interesting to know whether any 
    finiteness results hold in the symplectic non-K\"ahler category.
    \end{rem}
    
\bibliographystyle{amsplain}

\bigskip

\end{document}